\documentclass[11pt,reqno]{amsart}
\usepackage{amssymb,amsmath,amscd}
\usepackage{latexsym,bm,bbm,mathrsfs}
\usepackage{hyperref,graphicx, enumerate}
\usepackage{mathtools}
\usepackage{stmaryrd}
\usepackage[all,pdf]{xy}
\usepackage{graphicx}
\usepackage{tikz-cd}
\usepackage{mathtools}
\usepackage{color}
\graphicspath{ {./} }

\usepackage{stackrel}
\usepackage[left=30mm, right=30mm, top=30mm, bottom=30mm]{geometry}
\usepackage[shortlabels]{enumitem}
\usepackage{ulem}
\usepackage{float}
\usepackage{ dsfont }
\usepackage[left, pagewise, edtable]{lineno}
\usepackage{verbatim}%%% 여러 줄 주석 처리 \begin{comment}에 필요

\allowdisplaybreaks

\newtheorem{thm}{Theorem}[section]

\newtheorem*{quest*}{Question}

\newtheorem{theorem}[thm]{Theorem}
\newtheorem{cor}[thm]{Corollary}
\newtheorem{prop}[thm]{Proposition}
\newtheorem{lemma}[thm]{Lemma}

\newtheorem{remark}[thm]{Remark}

\newtheorem*{theorem*}{Theorem}

\newcommand{\bbf}{{\mathbb{F}}}

\newcommand{\bbq}{{\mathbb{Q}}}

\newcommand{\bbk}{{\mathbb{K}}}

\newcommand{\bbz}{{\mathbb{Z}}}

\newcommand{\cB}{\mathcal{B}}

\newcommand{\cN}{\mathcal{N}}

\newcommand{\fkA}{\mathfrak{A}}

\newcommand{\Gal}{\operatorname{Gal}}

\newcommand{\Fl}{\bbf_{\ell}}
\newcommand{\sep}{\operatorname{sep}}

\renewcommand{\Im}{\operatorname{Im}}
\newcommand{\Kst}{\bbk\left(s,t\right)}
\newcommand{\SL}[1]{\operatorname{SL}_{2}\left(\bbz/{#1}\bbz\right)}
\newcommand{\GL}[1]{\operatorname{GL}_{2}\left(\bbz/{#1}\bbz\right)}
\newcommand{\peu}{\left(\bbz/p^{e}\bbz\right)^{\times}}
\newcommand{\pu}{\left(\bbz/p\bbz\right)^{\times}}

\renewcommand{\S}[2]{\mathfrak{S}_{{#1}}\left({#2}\right)}

\newcommand{\Lint}{L_{p^{e} N}~\cap~L_{N \ell}}
\newcommand{\abs}[1]{{\left|{#1}\right|}}
\newcommand{\Disc}[1]{\operatorname{Disc}\left({#1}\right)}

\newcommand{\ab}{\operatorname{ab}}
\renewcommand{\ss}{\operatorname{ss}}
\newcommand{\cyc}{\operatorname{cyc}_{N}}

\newcommand{\Fp}{\bbf_{p}}

\newcommand{\alg}{{\operatorname{alg}}}

\newcommand{\cNintersec}{\cN_{p^{e},\cB_{N \ell},\ell}}

\usepackage{enumitem}

\title[]{The automorphism group of torsion points of an elliptic curve over a field of characteristic $\ge 5$}

\author{Bo-Hae Im}
\address{Department of Mathematical Sciences, KAIST, 291 Daehak-ro, Yuseong-gu, Daejeon, 34141, South Korea}
\email{bhim@kaist.ac.kr}

\author{Hansol Kim}
\address{Institute of Mathematics, Academia Sinica, 6F, Astronomy-Mathematics Building,
	No. 1, Sec. 4, Roosevelt Road, Taipei 10617, Taiwan}
\email{jawlang@as.edu.tw}

\thanks{
Bo-Hae Im was supported by Basic Science Research Program through the National Research Foundation of Korea(NRF) grant funded by the Korea government(MSIT)(NRF-2023R1A2C1002385). Hansol Kim was supported by NSTC grant funded by Taiwan government (MST) (No.~113-2811-M-001-068).}

\date{\today}
\subjclass[2010]{Primary 11, 11, Secondary 11}

\begin{document}
\maketitle

\begin{abstract}
	For a field $\bbk $ of characteristic $p\ge5$ containing $\Fp^{\alg}$ and the elliptic curve $E_{s,t}: y^{2} = x^{3} + sx + t$ defined over the function field $\Kst$ of two variables $s$ and $t$, we prove that for a non-negative positive integer $e$ and a positive integer $N$ which is not divisible by $p$, the automorphism group of the normal extension $\Kst\left(E_{s,t}\left[p^{e} N\right]\right)$ over $\Kst$ is isomorphic to $\left(\bbz/p^{e}\bbz\right)^{\times} \times \SL{N}$.
\end{abstract}

\section{Introduction}\label{sec:intro}

Elliptic curves are a central research topic in number theory and arithmetic geometry, playing a vital role in various fields, including algebraic geometry. Many Diophantine equations can be connected to elliptic curves, such as those arising in Fermat's Last Theorem and the congruence number problem. Moreover, when the Mordell-Weil theorem (\cite[Ch.VIII, Theorem~6.7]{Silverman}) asserts that the group of rational points of an elliptic curve is a finitely generated abelian group. Thus, finding only a finite number of generators yields a complete characterization of all rational points on the elliptic curve.

There are various approaches to compute the rank and the torsion subgroup of an elliptic curve and one of them is to apply Galois representation. For an elliptic curve $\mathscr{E}$ over a field $F$ and a positive integer $N$ which is not divisible by the characteristic of $F$, it is known that $\mathscr{E}\left[N\right] \cong \left(\bbz/N\bbz\right)^{2}$ and the field of definition of $N$-torsion points $F\left(\mathscr{E}\left[N\right]\right)$ over $F$ is Galois. Given a basis of $\cB_{N} = \left\{P_{N}, Q_{N}\right\}$ of $\mathscr{E}\left[N\right]$, we have an injective group homomorphism $r_{\cB_{N}}: \Gal\left(F\left(\mathscr{E}\left[N\right]\right)/F\right) \hookrightarrow \GL{N}$ defined by $$
	r_{\cB_{N}}\left(\sigma\right)\begin{pmatrix}P_{N}\\Q_{N}\end{pmatrix} = \begin{pmatrix}P_{N}^{\sigma}\\Q_{N}^{\sigma}\end{pmatrix}.
$$ For positive divisors $m$ and $n$ of $N$ such that $m \mid n$, we have that $\mathscr{E}\left(F\right)\left[N\right] \supseteq \bbz/m\bbz \oplus \bbz/n\bbz$ if and only if $$
	\operatorname{Im} \left(r_{\cB_{N}}\right)
	\subseteq
	\left\{
		\begin{pmatrix}a&b\\c&d\end{pmatrix} \in \GL{N} :
		\begin{aligned}\left(a,b\right) & \equiv \left(1,0\right) \pmod{m}\\\left(c,d\right) & \equiv \left(0,1\right) \pmod{n}\\
		\end{aligned}
	\right\}.
$$
\text{(See \cite[Lemma~3.2]{IK22}).}

The mod-$N$ representation $r_{\cB_{N}}$ is widely studied for the cases when $F = \bbq$. For example, Serre's open image theorem (\cite{Serre72}) states that, for each non-CM elliptic curve over $\bbq$, the mod-$\ell$ Galois representations are surjective for almost all primes $\ell$. Moreover, it is conjectured that the upper bound of $\ell$ such that mod-$\ell$ representation is reducible  is uniform, i.e., such an upper bound does not depend on non-CM elliptic curves.

Let $\bbk$ be a field of characteristic $p \ge 5$ containing a fixed algebraic closure $\Fp^{\alg}$ of $\Fp$, $e$ a non-negative integer, and $N$ a positive integer which is not divisible by $p$. The main goal of this paper is to characterize the automorphism group of the $p^{e} N$-torsion group $E\left[p^{e} N\right]$ of the elliptic curve $$
	E_{s,t}: y^{2} = x^{3} + sx + t
$$ defined over the function field $\Kst$ of two variables $s$ and $t$. More explicitly, this paper computes the automorphism group $\fkA\left(\Kst\left(E_{s,t}\left[p^{e} N\right]\right)/\Kst\right)$ where $\fkA\left(F'/F\right)$ denotes the automorphism group for a normal extension $F'/F$. It is known (\cite[Ch.III, Corollary~6.4]{Silverman} and \cite[Ch.V, Theorem~4.1~(b)]{Silverman}) that $$
	E_{s,t}\left[p^{e} N\right] \cong \bbz/p^{e}\bbz \oplus \left(\bbz/N\bbz\right)^{2}.
$$ Since the automorphisms of the normal extension $\Kst\left(E_{s,t}\left[p^{e} N\right]\right)/\Kst$ fix the order of the points in $E_{s,t}\left[p^{e} N\right]$, the automorphism group is a subgroup of $\peu \times \GL{N}$. More precisely, for a basis of $\cB_{N} = \left\{P_{N}, Q_{N}\right\}$ of $E_{s,t}\left[N\right]$ and a generator $R_{e}$ of $E_{s,t}\left[p^{e}\right]$, we have an injective group homomorphism $$
	\left(s_{e}, r_{\cB_{N}}\right): \fkA\left(\Kst\left(E_{s,t}\left[p^{e} N\right]\right)/\Kst\right) \hookrightarrow \peu \times \GL{N}
$$ defined by; for each $\sigma \in \fkA\left(\Kst\left(E_{s,t}\left[p^{e} N\right]/\Kst\right)\right)$, \begin{equation} \label{eqn:repn}
	s_{e}\left(\sigma\right)R_{e} = R_{e}^{\sigma}
	\text{ and }
	r_{\cB_{N}}\left(\sigma\right)\begin{pmatrix}P_{N}\\Q_{N}\end{pmatrix} = \begin{pmatrix}P_{N}^{\sigma}\\Q_{N}^{\sigma}\end{pmatrix}.
\end{equation} Since $\bbk$ contains a primitive $N$th root of unity, via the Weil paring, the image of $r_{\cB_{N}}$ is contained in $\SL{N}$ (\cite[Ch.III.8]{Silverman}). 
It is known that $\rho_{e, \cB_{N}}$ is an isomorphism if $e=0$ or $N=1$ as follows.

\begin{theorem}[{\cite[Theorem~1.1]{IK24}}]\label{thm:fact:pe}
	For a positive integer $e$, the group homomorphism
	$$s_{e}: \fkA\left(\Kst\left(E_{s,t}\left[p^{e}\right]\right)/\Kst\right) \to \peu$$ is an isomorphism.
\end{theorem}

\begin{theorem}[{\cite[Ch.IV, 6.8~Corollary~1]{CF}}, {\cite[Section~1 (right after (1))]{IK24}}]\label{thm:fact:N}
	For a positive integer $N$ which is not divisible by $p$ and a basis $\cB_{N}$ of $E_{s,t}\left[N\right]$, the group homomorphism
	$$r_{\cB_{N}}: \Gal\left(\Kst\left(E_{s,t}\left[N\right]\right)/\Kst\right) \to \SL{N}$$ is an isomorphism.
\end{theorem}

We denote the group homomorphism $\left(s_{e}, r_{\cB_{N}}\right)$ which depends only on a positive integer~$e$ and a basis $\cB_{N}$ of $E_{s,t}\left[N\right]$ by $$
	\rho_{e, \cB_{N}} := \left(s_{e}, r_{\cB_{N}}\right) : \fkA\left(\Kst\left(E_{s,t}\left[p^{e} N\right]\right)/\Kst\right) \hookrightarrow \peu \times \SL{N}.
$$

Although Theorem~\ref{thm:fact:pe} and Theorem~\ref{thm:fact:N} assert that each component of $\rho_{e, \cB_N}$ is individually an isomorphism, it does not follow automatically that the combined map $\rho_{e, \cB_N}$ is surjective. Indeed, even if the projections of a subgroup to each factor are surjective, the subgroup itself may not be the full product. (See Remark~\ref{rmk-exp} for more explanation.) This subtlety is crucial and necessitates a separate and careful argument to prove the following main results.

Now, we are ready to state our main result which proves that $\rho_{e, \cB_{N}}$ is an isomorphism for given~$\cB_{N}$ and $e$, thereby synthesizing and unifying the separated conclusions of Theorem~\ref{thm:fact:pe} and Theorem~\ref{thm:fact:N}, and also determine the automorphism group of torsion points.

We denote the image of $\rho_{e, \cB_{N}}$ by $G_{e, \cB_{N}}$, i.e., $$
	G_{e, \cB_{N}}:=\rho_{e, \cB_{N}}\left(\fkA\left(\Kst\left(E_{s,t}\left[p^{e} N\right]\right)/\Kst\right)\right).
$$

\begin{theorem}\label{thm:main}
	Let $\bbk$ be a field of characteristic $p \ge 5$ containing $\Fp^{\alg}$, $N$ a positive integer which is not divisible by $p$, and $e$ a non-negative integer. For the elliptic curve $E_{s,t}: y^{2} = x^{3} + sx + t$ defined over $\Kst$ and a basis $\cB_{N}$ of $E_{s,t}\left[N\right]$, we have that $$
		G_{e, \cB_{N}} = \peu \times \SL{N}.
	$$
\end{theorem}
\begin{remark}\label{rmk-exp}
	To emphasize the subtlety, consider the subgroup $G \subseteq \left(\bbz/5\bbz\right)^{\times} \times \SL{2}$ generated by $\left(\overline{2},\begin{pmatrix}\overline{0}&\overline{1}\\\overline{1}&\overline{0}\end{pmatrix}\right)$ and $\left(\overline{1},\begin{pmatrix}\overline{1}&\overline{1}\\\overline{0}&\overline{1}\end{pmatrix}\right)$. Both projections of $G$  to $\left(\bbz/5\bbz\right)^{\times}$ and to $\SL{2}$ are surjective. However,  $G $ is strictly smaller than the full product $\left(\bbz/5\bbz\right)^{\times} \times \SL{2}$, since $G$ does not contain the element $\left(\overline{1}, \begin{pmatrix}\overline{0}&\overline{1}\\\overline{1}&\overline{0}\end{pmatrix}\right)$. This shows us that the surjectivity of individual projections  given in Theorem~\ref{thm:fact:pe} and Theorem~\ref{thm:fact:N} does not guarantee that the combined map $\rho_{e, \cB_N}$ is surjective.
\end{remark}

Theorem~\ref{thm:main} enables the following computation of the automorphism group of the $p^{e} N$-torsion subgroup of $E_{s,t}/K\left(s,t\right)$ in a general setting where the ground field $K$ is not assumed to contain $\Fp^{\alg}$, by exploring the relation between the Weil pairing and the determinant of the associated Galois representation.

\begin{cor}\label{cor:gen}
	Let $K$ be an arbitrary field of characteristic $p\geq 5$, $N$ a positive integer which is not divisible by $p$, and $e$ a non-negative integer. For the elliptic curve $E_{s,t}: y^{2} = x^{3} + sx + t$ defined over $K\left(s,t\right)$ and any basis $\cB_{N}$ of $E_{s,t}\left[N\right]$, we have that $$
		G_{e, \cB_{N}} = \left\{\left(u,A\right) \in \peu \times \GL{N}: \det A \in \Im\left(\cyc\right)\right\},
	$$ where $\cyc$ is the mod-$N$ cyclotomic character of the absolute Galois group of $K$.
\end{cor}

The structure of our paper is as follows: In Section~\ref{sec:gen_p}, we present a statement that implies Theorem~\ref{thm:main}, along with clues that guide its proof. Section~\ref{sec:N=S} uses these clues to complete the proof of the statement in Section~\ref{sec:gen_p}. In Section~\ref{sec:thm_and_cor}, we provide complete proofs of Theorem~\ref{thm:main} and Corollary~\ref{cor:gen}.

\

Throughout this paper,  all the fields under consideration--including the ground fields $\bbk$ and~$K$--have positive characteristic, denoted by~$p$, unless otherwise specified. We further assume that $p\geq 5$.

\section{A Key inductive argument toward the proof of Theorem~\ref{thm:main}}\label{sec:gen_p}
The key step to prove Theorem~\ref{thm:main} is to establish the intermediate induction steps on $N$. More precisely, we prove Proposition~\ref{prop:induction}, whose proof is given in Section~\ref{sec:thm_and_cor}.

\begin{prop}\label{prop:induction}
	Let $\ell$ be a prime and $N$ a positive integer such that $p\nmid N \ell$. If $G_{e, \cB_{N}} = \peu \times \SL{N}$ for a given basis $\cB_{N}$ of $E_{s,t}\left[N\right]$, then for any basis $\cB_{N \ell}$ of $E_{s,t}\left[N \ell\right]$, $$
		G_{e, \cB_{N \ell}} = \peu \times \SL{N \ell}.
	$$
\end{prop}

Proposition~\ref{prop:induction} reduces Theorem~\ref{thm:main} for all positive integers $N$ into Theorem~\ref{thm:main} for $N=1$,  which follows from Theorem~\ref{thm:fact:pe}.

Now, we explain the strategy for proving Proposition~\ref{prop:induction}. For convenience, let $$L_{1} := \Kst,$$ and  for each integer $m\ge2$, let $L_{m}$ be the normal extension $\Kst\left(E_{s,t}\left[m\right]\right)$ of $\Kst$. For a given $N$ and a positive divisor $n$ of $N$, let \begin{equation}\label{eqn:subgroup_SL}
	\S{N}{n} := \left\{A \in \SL{N}: A \equiv \begin{pmatrix}1&0\\0&1\end{pmatrix}\pmod{n}\right\},
\end{equation} 
which is a normal subgroup of $\SL{N}$.
The main ingredient for proving Proposition~\ref{prop:induction} is Proposition~\ref{prop:tool}, which relies on  the following Lemma~\ref{lem:extn_deg} along with Theorem~\ref{thm:fact:N}.

\begin{lemma}\label{lem:extn_deg}
	Let $F_{0}$ be a field, $F_{1}/F_{0}$ a finite Galois extension, and $F_{2}/F_{0}$ a finite extension. Then, the following holds for the degrees and separable degrees:
	$$		\left[F_{1}F_{2} : F_{0}\right] = \frac{\left[F_{1}: F_{0}\right] \left[F_{2} : F_{0}\right] }{\left[F_{1} \cap F_{2} : F_{0}\right]}
		\text{ and }
		\left[F_{1}F_{2} : F_{0}\right]_{\sep} = \frac{\left[F_{1}: F_{0}\right] \left[F_{2} : F_{0}\right]_{\sep} }{\left[F_{1} \cap F_{2} : F_{0}\right]}.
	$$
\end{lemma}
\begin{proof}
	Since $F_{1}/F_{0}$ is Galois, both $F_{1}F_2/F_{2}$ and $F_1/F_1\cap F_2$ are Galois, so we have
	\begin{align*}
		\left[F_{1}F_{2} : F_{0}\right] & = \left[F_{1}F_{2} : F_{2}\right] \left[F_{2} : F_{0}\right] = \left[F_{1} : F_{1} \cap F_{2}\right] \left[F_{2} : F_{0}\right] = \frac{\left[F_{1}: F_{0}\right] \left[F_{2} : F_{0}\right] }{\left[F_{1} \cap F_{2} : F_{0}\right]},
	\end{align*}
	and
	\begin{align*}
		\left[F_{1}F_{2} : F_{0}\right]_{\sep}= \left[F_{1}F_{2} : F_{2}\right] \left[F_{2} : F_{0}\right]_{\sep} = \left[F_{1} : F_{1} \cap F_{2}\right] \left[F_{2} : F_{0}\right]_{\sep} = \frac{\left[F_{1} : F_{0}\right]\left[F_{2} : F_{0}\right]_{\sep}}{\left[F_{1} \cap F_{2}:F_{0}\right]}.
	\end{align*}
\end{proof}

For a given basis $\cB_{N \ell}$ of $E_{s,t}\left[N \ell\right]$, we let  $$
	\cNintersec :=r_{\cB_{N \ell}}\left(\Gal\left(L_{N \ell} / \Lint\right)\right).
$$

\begin{prop}\label{prop:tool}
	Let $\ell$ be a prime and $N$ a positive integer such that $p\nmid N \ell$.	If $G_{e, \cB_{N}} = \peu \times \SL{N}$ for a basis $\cB_{N}$ of $E_{s,t}\left[N\right]$, then, for a basis $\cB_{N \ell}$ of $E_{s,t}\left[N \ell\right]$, we have the following: \begin{enumerate}[\normalfont (a)]
		\item \label{normal} $\cNintersec$ is a normal subgroup of $\SL{N \ell}$. 
		\item \label{cyc} $\cNintersec$ is a normal subgroup of $\S{N \ell}{N}$ defined in~\eqref{eqn:subgroup_SL}, and $\S{N \ell}{N} / \cNintersec$ is cyclic.
		\item \label{goal} If $\cNintersec = \S{N \ell}{N}$, then $G_{e, \cB_{N \ell}} = \peu \times \SL{N \ell}$.
	\end{enumerate}
\end{prop}

\begin{proof}
	It is easy to show that $L_{p^{e} N \ell}$ is the compositum of two normal extensions $L_{N \ell}$ and $L_{p^{e} N}$ over $L_{N}$. By Theorem~\ref{thm:fact:N}, we have the following diagrams, where the dotted markings on the group side indicate statements that need to be proved:
	\begin{figure}[H]
		\centering
		% https://tikzcd.yichuanshen.de/#N4Igdg9gJgpgziAXAbVABwnAlgFyxMJZABgBoBGAXVJADcBDAGwFcYkQAZAfWADkACNAD1gMAL5iQY0uky58hFACYK1Ok1btuwADo6YjRv16Tps7HgJFypJWoYs2iEHo5YwOKTJAYLC66QALPYaTpw85GL8ALz8egDScJ5mPnKWisg2AMwhjlo8Jl7m8lYoNsS5ms7aegZGAsKiEkWpfqXIgao0DlUuOowwAGY4esAAkjxKYnoATlgA5gAWIzqm3r4lGVm2lWF6AMa8PLWGpAcAQsf6hsbSjeJrxelE2zndoex6AMq61-VifEerU2LyCu0+Oi+HF+dVuUjUMCg83gRFAgxmEAAtkhtiAcBAkJ0QIsYPQoOxIGA2Cl0VjCTR8UgVMTSeTnJTqd5adjEMzGYgyCyyRSCJy0RieYL+QBWGgk4Xs0UtblMhkExA2IVs8BKmkSpCamVy1kiqnK-WIACcaqQAA5jQqdWa9XTEPa8eqAOwO7Uc82u70epAANh9prFIBVVptiEDACMYGA2YFrVrwyAaIx3GEoPQ4CTye88s49DgYAAPHDASAzTFMIFR1P890JpNIAC0wcF+awww7mqzZucufziIzacVZsz9ATjAACml-M45ktPEXeqWK1Wa3XGA2Le7+fHE2z2ymw5O2Jns+wRwXxz09joy5XqxBa0x+PQk-x9gBPfYs32SRp1nBc2kUEAV2WeExCAA
\begin{tikzcd}
								& {\color{lightgray}L_{p^{e} N \ell}}							 &								&																																								  &					 \\
{\color{lightgray}L_{p^{e} N} \arrow[lightgray, ru, no head]} &											  & L_{N \ell} \arrow[lightgray, lu, no head] &																																								  & \left\{I_{2}\right\} \\
								& \Lint \arrow[lightgray, lu, no head] \arrow[ru, no head] &								& {\cN_{\ell,\cB_{N\ell},p^{e}}} \arrow[ru, no head]																											  &					 \\
								& L_{N} \arrow[u, no head]					 &								& \S{N\ell}{N} \arrow[u, no head] \arrow[u, "\text{(b) normal and cyclic}"', no head, dashed, bend right=49]														  &					 \\
								& L_{1} = \Kst \arrow[u, no head]			  &								& \SL{N\ell} \arrow[u, no head] \arrow[uu, "\text{(a) normal}", no head, dashed, bend left=49] &					 
\end{tikzcd}
		\caption{The diagrams of subfields of $L_{N \ell}$ over $L_{1}$ and their corresponding Galois groups of $L_{N \ell}$ over a field containing $L_{1}$}
		\label{fig:fld_extn}
	\end{figure}
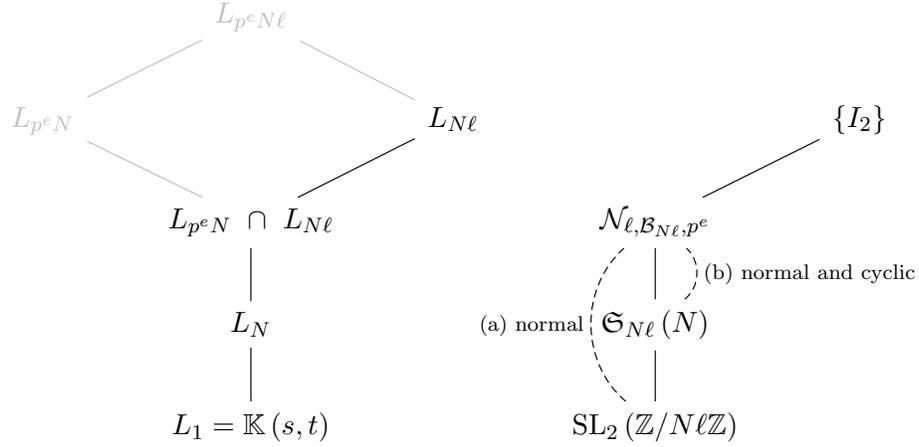
	
	The part \ref{normal} follows from the fact that $\Lint$ is the intersection of two normal extensions $L_{N \ell}$ and $L_{p^{e} N}$ over $L_{1}$. 
	
	To prove \ref{cyc}, we first observe that $\cNintersec$ is a normal subgroup of $\S{N \ell}{N}$, since $\Lint$ equals the intersection of two normal extensions $L_{N \ell}$ and $L_{p^{e} N}$ over $L_{N}$. The group $\S{N \ell}{N} / \cNintersec$ is isomorphic to the Galois group of $\Lint $ over $ L_{N}$, and it is a quotient group of the cyclic automorphism group of the normal extension $L_{p^{e} N}$ over $L_{N}$ by Theorem~\ref{thm:fact:pe}, so it is also cyclic.

	Now, for \ref{goal}, we first note that by Lemma~\ref{lem:extn_deg}, we have that \begin{align}\label{eqn:comparing_extn_deg}
		\notag\left|G_{e, \cB_{N \ell}}\right| & = \left[L_{p^{e} N \ell}:L_{1}\right]_{\sep} = \frac{\left[L_{N \ell}:L_{1}\right]\left[L_{p^{e} N}:L_{1}\right]_{\sep}}{\left[\Lint:L_{1}\right]} 
		= \frac{\left[L_{N \ell}:L_{1}\right] \left[L_{p^{e} N}:L_{1}\right]_{\sep} }{\left[\Lint:L_{N}\right] \left[L_{N}:L_{1}\right]} \\
		 &= \frac{\left[L_{N \ell}:L_{1}\right] \left[L_{p^{e} N}:L_{1}\right]_{\sep} }{\left[\Lint:L_{N}\right] \left[L_{N}:L_{1}\right]_{\sep}} 
		= \frac{\left[L_{N \ell}:L_{1}\right] \left[L_{p^{e} N}:L_{N}\right]_{\sep} }{\left[\Lint:L_{N}\right]}.
	\end{align} 
	We also have $\left[L_{N \ell}:L_{1}\right] = \left| \SL{N \ell}\right|$ by Theorem~\ref{thm:fact:N}, and \begin{align*}
		\left[L_{p^{e} N}:L_{N}\right]_{\sep} = \left| \left\{\left(u,A\right) \in \peu \times \SL{N} : A = \begin{pmatrix}1&0\\0&1\end{pmatrix}\right\} \right| = \left| \peu \right|
	\end{align*} by the assumption that $G_{e, \cB_{N}} = \peu \times \SL{N}$. 
	
	If $\cNintersec = \S{N \ell}{N}$, then $\Lint = L_{N}$, since the extension $\Lint$ over $L_{N}$ is separable. In this case, we see by \eqref{eqn:comparing_extn_deg} above that $$
		\left|G_{e, \cB_{N \ell}}\right| = \left[L_{N \ell}:L_{1}\right] \left[L_{p^{e} N}:L_{N}\right]_{\sep} = \left|\SL{N \ell}\right| \left| \peu \right|.
	$$ Since $G_{e, \cB_{N \ell}} \subseteq \peu \times \SL{N \ell}$, we conclude that $G_{e, \cB_{N \ell}} = \peu \times \SL{N \ell}$.
\end{proof}

To prove Proposition~\ref{prop:induction} in Section~\ref{sec:thm_and_cor}, we apply Proposition~\ref{prop:tool}~\ref{goal} and demonstrate $\cNintersec = \S{N \ell}{N}$ referring to Figure~\ref{fig:fld_extn}, whose proof is divided into the following cases, each requiring substantially different arguments: \begin{enumerate}[{label=(\Alph*)}]
	\item $\ell \mid N$,
	\item $\ell \nmid N$, for $\ell\geq 5$, and
	\item $\ell \nmid N$, for $\ell \in \left\{2,3\right\}$.
\end{enumerate}

In Section~\ref{subsec:general}, we analyze the group structure of $\SL{N}$ and show that $\cNintersec = \S{N \ell}{N}$ for Case~(A) and Case~(B). In Section~\ref{subsec:2or3}, we prove $\cNintersec = \S{N \ell}{N}$ for Case~(C).

\section{Proof of $\cNintersec = \S{N \ell}{N}$}\label{sec:N=S}
\subsection{The cases when $\ell \ge 5$ or $\ell \mid N$}\label{subsec:general}

In this section, we prove $\cNintersec = \S{N \ell}{N}$  for the cases when $\ell \ge 5$ or $\ell \mid N$. To achieve this, we need to analyze the group structures of ~$\SL{N}$, as described in Lemma~\ref{lem:SLNN} and Lemma~\ref{lem:SLl_comm}.

\begin{lemma}\label{lem:SLNN} For a prime $\ell$ and a positive integer $N$, \begin{enumerate}[\normalfont (a)]
		\item\label{SL} if $\ell \nmid N$, then $\S{N \ell}{N}\cong \operatorname{SL}_{2}\left(\Fl\right),$ and 
		\item\label{vec} if $\ell \mid N$, then $\S{N \ell}{N}\cong \left(\bbz/\ell\bbz\right)^{3}$.
	\end{enumerate}
\end{lemma}
\begin{proof}
	\ref{SL} Suppose $\ell \nmid N$. We show that the mod-$\ell$ reduction $\pi: \S{N \ell}{N} \to \operatorname{SL}_{2}\left(\Fl\right)$ is an isomorphism using the Chinese remainder theorem. By this theorem, the only $A \in \ker \pi$ satisfying $A \equiv I_{2} \pmod{N}$ is $A = I_{2}$, where $I_2 \in \SL{N \ell}$ denotes the identity matrix. Hence, $\pi$ is injective. Furthermore, for any $B \in \operatorname{SL}_{2}\left(\Fl\right)$, the Chinese remainder theorem shows the existence of $A \in \SL{N \ell}$ such that $A \equiv B \pmod{\ell}$ and $A \equiv I_{2} \pmod{N}$. Hence, $\pi\left(A\right) = B$ and $\pi$ is surjective.

	\ref{vec} Suppose $\ell \mid N$. Consider the  group isomorphism $\nu: \operatorname{M}_{2}\left(\bbz/\ell\bbz\right) \to \operatorname{M}_{2}\left(N \bbz/N \ell \bbz\right)$ of additive groups of $2\times 2$ matrices, induced by the multiplication homomorphism $\bbz \to N\bbz$ given by $x\mapsto Nx$. Since $\ell \mid N$, the relation $\left(I_{2} + \nu\left(A\right)\right)\left(I_{2}+ \nu\left(B\right)\right) = I_{2}+\nu\left(A+B\right)$ holds for all matrices $A, B \in \operatorname{M}_{2}\left(\bbz/\ell\bbz\right)$, and we have $I_{2} + \nu\begin{pmatrix}x&y\\z&-x\end{pmatrix} \in \S{N \ell}{N}$ for all $x,y,z\in \bbz/\ell\bbz$. In other words, we have a group homomorphism $\gamma: \left(\bbz/\ell\bbz\right)^{3} \to \S{N \ell}{N}$ defined by \begin{equation}\label{gamma}
		\gamma\begin{pmatrix}x\\y\\z\end{pmatrix} = I_{2} + \nu \begin{pmatrix}x&y\\z&-x\end{pmatrix}.
	\end{equation}
	It is clear that $\gamma$ is injective. To show that $\gamma$ is surjective, consider any $X \in \S{N \ell}{N}$. Then, $X$ is expressed as $I_{2} + \nu\begin{pmatrix}a&b\\c&d\end{pmatrix}$ for some $a,b,c,d\in \bbz/\ell\bbz$. Since $\ell \mid N$, the condition that $\det\left(X\right) = 1$ implies $d = -a$. Therefore, $\gamma$ is surjective, so it is an isomorphism.
\end{proof}

\begin{lemma}\label{lem:SLl_comm}
\
	\begin{enumerate}[\normalfont (a)]
		\item\label{SLl_comm_small_ell} If $\ell \in \left\{2,3\right\}$, then there exists a unique proper normal subgroup $\cN$ of $\SL{\ell}$ such that the quotient group $\SL{\ell}/\cN$ is abelian, and in this case, the index $[\SL{\ell}:\cN]=\ell$.
		\item\label{SLl_comm_large_ell} If $\ell \ge 5$ is a prime, then no such proper normal subgroup of $\SL{\ell}$ with abelian quotient as in \ref{SLl_comm_small_ell} exists.
	\end{enumerate}
\end{lemma}
\begin{proof}
	If $\SL{\ell}/\cN$ is abelian for a proper normal subgroup $\cN$ of $\SL{\ell}$, then $\cN$ contains the commutator subgroup of $\SL{\ell}$. If $\ell\ge 5$, then by \cite[Theorem~1.9]{Gr}, the commutator subgroup of $\SL{\ell}$ is $\SL{\ell}$ itself. Therefore, there is no proper normal subgroup of $\SL{\ell}$.
	
	If $\ell \in \left\{2,3\right\}$, by direct computations, we can show that the index of the commutator subgroup in $\SL{\ell}$ is $\ell$. Therefore, $\cN$ is equal to the commutator subgroup and so it has index $\ell$.
\end{proof}

Now, we are ready to prove $\cNintersec = \S{N \ell}{N}$  for the cases when $\ell \ge 5$ or $\ell \mid N$.

\begin{prop}\label{prop:general}
	Let $\ell$ be a prime and $N$ a positive integer such that $p\nmid N \ell$ and assume either $\ell \ge 5$ or $\ell \mid N$. If $G_{e, \cB_{N}} = \peu \times \SL{N}$ for a basis $\cB_{N}$ of $E_{s,t}\left[N\right]$, then $\cNintersec = \S{N \ell}{N}$ referring to Figure~\ref{fig:fld_extn}, for any basis $\cB_{N \ell}$ of $E_{s,t}\left[N \ell\right]$.
\end{prop}

\begin{proof}
	If $\ell \ge 5$ and $\ell \nmid N$, by Proposition~\ref{prop:tool}~\ref{cyc}, Lemma~\ref{lem:SLNN}~\ref{SL}, and Lemma~\ref{lem:SLl_comm}~\ref{SLl_comm_large_ell}, we have $\cNintersec = \S{N \ell}{N}$.
	
	If $\ell \mid N$, then by Lemma~\ref{lem:SLNN}~\ref{vec}, $\S{N \ell}{N}$ is isomorphic to the column vector space $\left(\bbz/\ell\bbz\right)^{3}$ over $\Fl$. If we consider the isomorphism $\gamma: \left(\bbz/\ell\bbz\right)^{3} \to \S{N \ell}{N}$ given in \eqref{gamma}, to prove $\cNintersec = \S{N \ell}{N}$, it is enough to show that $\dim_{\Fl} \gamma^{-1}\left(\cNintersec\right) = 3$. Since $\S{N \ell}{N}$ is a normal subgroup of $\SL{N \ell}$, we have a group homomorphism $f: \SL{N \ell} \to \operatorname{Aut} \left(\S{N \ell}{N}\right) \cong \operatorname{GL}_{3}\left(\Fl\right)$ defined by $$
		A \left(\gamma\begin{pmatrix}x\\y\\z\end{pmatrix}\right) A^{-1}
		= \gamma \left(\left(f\left(A\right)\right)\begin{pmatrix}x\\y\\z\end{pmatrix}\right),
	$$ for $A\in \SL{N \ell}$. By Proposition~\ref{prop:tool}~\ref{normal}, the subspace $\gamma^{-1}\left(\cNintersec\right)$ is invariant under the image of $f$. Thus, to prove $\dim_{\Fl} \gamma^{-1}\left(\cNintersec\right) = 3$, by Proposition~\ref{prop:tool}~\ref{cyc}, it is enough to show that there is no $2$-dimensional subspace which is invariant under two matrices $$
		M_{1} := \begin{pmatrix}1&-1&0\\0&1&0\\2&-1&1\end{pmatrix} = f\begin{pmatrix}1&0\\1&1\end{pmatrix}
		\text{ and }
		M_{2} := \begin{pmatrix}1&0&1\\-2&1&-1\\0&0&1\end{pmatrix} = f \begin{pmatrix}1&1\\0&1\end{pmatrix}.
	$$ Then, $\dim_{\Fl} \gamma^{-1}\left(\cNintersec\right) = 3$. We note that the characteristic polynomial of $M_{1}$ is $\left(T-1\right)^{3}$. To proceed the proof, let $x_{1} = \begin{psmallmatrix}0\\0\\1\end{psmallmatrix}$, $x_{2} = \begin{psmallmatrix}1\\0\\1\end{psmallmatrix}$, and $x_{3} = \begin{psmallmatrix}1\\1\\1\end{psmallmatrix}$.

	If $\ell \ne 2$, the minimal polynomial of $M_{1}$ is $\left(T-1\right)^{3}$. Therefore, $M_{1}$ is similar to $\begin{pmatrix}1&1&0\\0&1&1\\0&0&1\end{pmatrix}$, which implies that there is only one $2$-dimensional subspace of $\Fl^{3}$ invariant under $M_{1}$. Indeed, $\left\langle x_{1}, x_{2} \right\rangle$ is the $2$-dimensional invariant subspace under $M_1$. However, this subspace is not invariant under $M_{2}$, since $M_{2}x_{1} = \begin{psmallmatrix}1\\-1\\1\end{psmallmatrix} \notin \left\langle x_{1}, x_{2} \right\rangle$.

	If $\ell=2$, the minimal polynomial of $M_{1}$ is $\left(T-1\right)^{2}$. Therefore, $M_{1}$ is similar to $\begin{pmatrix}1&1&0\\0&1&0\\0&0&1\end{pmatrix}$, which implies that there are exactly two $2$-dimensional subspaces of $\Fl^{3}$ that are invariant under~$M_{1}$. Indeed, $\left\langle x_{1}, x_{2} \right\rangle$ and $\left\langle x_{2}, x_{3} \right\rangle$ are those $2$-dimensional invariant subspaces. However, these spaces are not invariant under $M_{2}$, as $M_{2} x_{2}= \begin{psmallmatrix}0\\1\\1\end{psmallmatrix}$ does not belong to either $\left\langle x_{1}, x_{2} \right\rangle$ or~$\left\langle x_{2}, x_{3} \right\rangle$.
\end{proof}

\subsection{The cases when $\ell \in \left\{2,3\right\}$ does not divide $N$}\label{subsec:2or3}

\subsubsection{The main ingredient upon a discrete valuation for the proof}

For the cases when $\ell \in \left\{2,3\right\}$ and $\ell$ does not divide $N$, the proofs of $\cNintersec = \S{N \ell}{N}$ referring to Figure~\ref{fig:fld_extn} follow a significantly different approach from those for $\ell\geq 5$ or $\ell \mid N$. In these cases, Proposition~\ref{prop:2or3} upon valuations becomes essential, whose proof is built upon Lemma~\ref{lem:SL_comm}, Lemma~\ref{lem:cyc_cubic}, and Lemma~\ref{lem:ind_hyp} presented below.

\begin{lemma}\label{lem:SL_comm}
	Let $\ell$ be a prime and $N$ be a positive integer such that $\ell \nmid N$. Then, the group $\SL{N}$ does not contain a normal subgroup of index $\ell$.
\end{lemma}
\begin{proof}
	If $N=1$, it is trivially true. Let $N>1$ and suppose that there is a normal subgroup $\cN$ of index $\ell$ in $\SL{N}$. Since $\ell$ is a prime, $\SL{N}/ \cN$ is abelian and is isomorphic to a quotient group of $\left(\SL{N}\right)^{\ab}$, where $\left(\SL{N}\right)^{\ab}$ denotes the abelianization of $\SL{N}$. Let $N = \prod_{i=1}^{n} q_{i}^{e_{i}}$ be a prime factorization of $N$. Since $\left(\SL{N}\right)^{\ab} \cong \prod_{i=1}^{n} \left(\SL{q_{i}^{e_{i}}}\right)^{\ab}$ by the Chinese remainder theorem and $\ell \ne q_{i}$ for all $i$, it is enough to show that for each $q:=q_{i}$, $\left|\left(\SL{q^{e}}\right)^{\ab}\right| = q^{k}$ for some integer $k\ge 0$. Using mathematical induction, this result follows from Lemma~\ref{lem:ab} below: For $e\ge2$, by considering $\phi$ in Lemma~\ref{lem:ab} as the mod-$q^{e-1}$ reduction map, $\SL{q^{e}} \twoheadrightarrow \SL{q^{e-1}}$, along with Lemma~\ref{lem:SLNN}~\ref{vec}, we deduce that $\abs{\left(\SL{q^{e}}\right)^{\ab}}$ divides $q^{3} \abs{\left(\SL{q^{e-1}}\right)^{\ab}}$. On the other hand, Lemma~\ref{lem:SLl_comm} shows that the order of $\left(\SL{q}\right)^{\ab}$ is either $1$ or $q$. In conclusion, $\left|\left(\SL{q^{e}}\right)^{\ab}\right| = q^{k}$ for some integer $k\ge 0$.	
\end{proof}

\begin{lemma}\label{lem:ab}
		For a surjective group homomorphism $\phi: G \twoheadrightarrow H$, $\abs{G^{\ab}}$ divides $\abs{\ker \phi}\abs{H^{\ab}}$.
	\end{lemma}\begin{proof}
		Let $G'$ and $H'$ denote the commutator subgroups of $G$ and $H$, respectively. Since $\phi^{-1}\left(\phi\left(G'\right)\right) = G' \ker\phi$, it follows that $$
			\abs{G^{\ab}} = \left[G:G'\right] = \left[G: \phi^{-1}\left(\phi\left(G'\right)\right) \right] \left[ G' \ker\phi :G'\right].
		$$ Moreover, the surjectivity of $\phi$ implies that the induced group homomorphism $G/\phi^{-1}\left(\phi\left(G'\right)\right) \to H/\phi\left(G'\right) = H/H'$ is an isomorphism. Hence, we have: $$
			\abs{G^{\ab}} = \left[G: \phi^{-1}\left(\phi\left(G'\right)\right) \right] \left[ G' \ker\phi :G'\right] = \left[H: H' \right] \left[ \ker\phi : \ker\phi \cap G'\right] ~~\bigg|~~ \abs{H^{\ab}} \abs{\ker \phi}.
	$$
	\end{proof}

\begin{lemma}\label{lem:cyc_cubic} 
	Let $\ell \in\{2,3\}$ and $F$ be a field whose characteristic is different from $\ell$. If $\ell = 3$, we assume that $F$ contains a primitive third root of unity $\omega$. For non-zero $\beta_{1}, \beta_{2} \in F$, if $F\left(\sqrt[\ell]{\beta_{1}}\right) = F\left(\sqrt[\ell]{\beta_{2}}\right)$, then there exists an $\epsilon \in \left\{\pm1\right\}$ such that $\sqrt[\ell]{\beta_{1}^{\epsilon}\beta_{2}} \in F$.
\end{lemma}
\begin{proof}
	Let $F' := F\left(\sqrt[\ell]{\beta_{1}}\right) = F\left(\sqrt[\ell]{\beta_{2}}\right)$.
	First, let $\ell =2$. If $F' = F$, then both $\beta_{1}$ and $\beta_{2}$ are squares in $F$, and thus, $\beta_{1}\beta_{2}$ is also a square in $F$. Next, if $F'\ne F$, we have $a,b \in F$ such that $\sqrt{\beta_{2}} = a + b\sqrt{\beta_{1}}$. This gives the equation, $$
		\left(a^{2}+\beta_{1}b^{2}\right) + 2ab\sqrt{\beta_{1}}= \beta_{2} \in F.
	$$ For this to hold, we must have $ab=0$, since $\operatorname{char}\left(F\right)\neq \ell=2$. If $b=0$, then $\beta_{2} = a^{2}$, so $\sqrt{\beta_{2}} \in F$, which contradicts that $F\left(\sqrt{\beta_{2}}\right) \ne F$. Therefore, $a= 0$ and $\sqrt{\beta_{1}\beta_{2}} = b\beta_{1} \in F$.

	Next, let $\ell=3$. If $F' = F$, then both $\beta_{1}$ and $\beta_{2}$ are cubes in $F$, and thus, $\beta_{1}\beta_{2}$ is also a cube in $F$. If $F' \ne F$, then the extension $F'$ over $F$ is cyclic of order $3$, since $\omega \in F$. For a generator $\sigma \in \Gal\left(F'/F\right)$, we have $\epsilon_{1}, \epsilon_{2} \in \left\{\pm1\right\}$ such that $\sqrt[3]{\beta_{1}}^{\sigma} = \sqrt[3]{\beta_{1}} \omega^{\epsilon_{1}}$ and $\sqrt[3]{\beta_{2}}^{\sigma} = \sqrt[3]{\beta_{2}} \omega^{\epsilon_{2}}$. Let $\epsilon := -\epsilon_{2}/\epsilon_{1}$. Then, $\sigma$ fixes $\sqrt[3]{\beta_{1}}^{\epsilon}\sqrt[3]{\beta_{2}}$, which implies that $\beta_{1}^{\epsilon}\beta_{2}$ is a cube in $F$.
\end{proof}

\begin{lemma}\label{lem:ind_hyp}
	Let $m$ and $n$ be two relatively prime positive integers such that $p\nmid mn$. Then,
	\begin{enumerate}[\normalfont (a)]
		\item $L_{m} \cap L_{n} = L_{1}$ and
		\item if $G_{e, \cB_{n}} = \peu \times \SL{n}$ for a basis $\cB_{n}$ of $E_{s,t}\left[n\right]$, then $L_{n} \cap L_{p^{e}} = L_{1}$.
	\end{enumerate} 
\end{lemma}
\begin{proof}
	(a) Since $L_{mn} = L_{m}L_{n}$, Lemma~\ref{lem:extn_deg} and Theorem~\ref{thm:fact:N} imply that $$
		\left|\SL{mn}\right|
		= \left[L_{mn}: L_{1}\right]
		= \frac{\left[L_{m}: L_{1}\right] \left[L_{n}: L_{1}\right] }{\left[L_{m} \cap L_{n}: L_{1}\right]} 
		= \frac{\left|\SL{m}\right| \left|\SL{n}\right| }{\left[L_{m} \cap L_{n}: L_{1}\right]}
	$$ and $L_{m} \cap L_{n} = L_{1}$.

	(b) If $G_{e, \cB_{n}} = \peu \times \SL{n}$, Lemma~\ref{lem:extn_deg}, Theorem~\ref{thm:fact:pe}, and Theorem~\ref{thm:fact:N} imply that $$
		\left| \peu \times \SL{n}\right|
		= \left[L_{n p^{e}}: L_{1}\right]_{\sep}
		= \frac{\left[L_{n}: L_{1}\right] \left[L_{p^{e}}: L_{1}\right]_{\sep} }{\left[L_{n} \cap L_{p^{e}}: L_{1}\right]} 
		= \frac{\left|\SL{n}\right| \left|\peu\right| }{\left[L_{n} \cap L_{p^{e}}: L_{1}\right]}
	$$ and $L_{n} \cap L_{p^{e}} = L_{1}$.
\end{proof}

We are ready to prove Proposition~\ref{prop:2or3}, which is the main tool for proving $\cNintersec = \S{N \ell}{N}$ referring to Figure~\ref{fig:fld_extn} for the cases when $\ell \in \left\{2,3\right\}$ does not divide $N$.

\begin{prop}\label{prop:2or3} 
	Let $\ell \in \left\{2,3\right\}$ and $N$ be a positive integer such that $p\nmid \ell N$ and $\ell \nmid N$. Assume that $G_{e, \cB_{N}} = \peu \times \SL{N}$ for a basis $\cB_{N}$ of $E_{s,t}\left[N\right]$. If \begin{enumerate}[\normalfont (a)]
		\item \label{easy} $\ell \nmid p-1$, or
		\item \label{goal_2or3} there exist a discrete valuation $v$ on $L_{1}$ and $\lambda, \mu \in L_{1}$ such that $\sqrt[\ell]{\lambda} \in L_{\ell}$, $\sqrt[\ell]{\mu} \in L_{p}$, $\sqrt[\ell]{\lambda} \notin L_{1}$, and $\ell$ divides $v\left(\lambda\right)$ but does not divide $v\left(\mu\right)$,
	\end{enumerate} then $\cNintersec = \S{N \ell}{N}$ referring to Figure~\ref{fig:fld_extn}.
\end{prop}
\begin{proof}
	Assume (a) is satisfied. Then, if $\ell=2$, then $\ell\mid p-1$, so $\ell=3$. By the assumption that $G_{e, \cB_{N}} = \peu \times \SL{N}$, the separable degree of $L_{p^{e} N} / L_{N}$ is $\left|\left(\bbz/p^{e}\bbz\right)^{\times}\right| = \left(p-1\right)p^{e-1}$, and hence, $[L_{p^{e} N}:L_{N}]$ is not divisible by $\ell =3$. On the other hand, by Proposition~\ref{prop:tool}~\ref{cyc}, Lemma~\ref{lem:SLNN}~\ref{SL}, and Lemma~\ref{lem:SLl_comm}~\ref{SLl_comm_small_ell}, the index of $\cNintersec$ in $\S{N \ell}{N}$ is $1$ or $\ell$. Since this index divides $[L_{p^{e} N}:L_{N}]$ which is not divisible by $\ell$, it follows that $\cNintersec = \S{N \ell}{N}$.
	
	Next, assume that \ref{goal_2or3} holds and $\ell \mid p-1$. We suppose $\cNintersec \ne \S{N \ell}{N}$ and show that \begin{equation}\label{eqn:trinity}
		L_{N}\left(\sqrt[\ell]{\lambda}\right) = \Lint = L_{N}\left(\sqrt[\ell]{\mu}\right).
	\end{equation}
	Then, $\sqrt[\ell]{\lambda^{\epsilon}\mu} \in L_{N}$ for some $\epsilon \in \left\{\pm1\right\}$ by Lemma~\ref{lem:cyc_cubic}. Since $\ell \mid v\left(\lambda\right)$ and $\ell \nmid v\left(\mu\right)$, we have that $\sqrt[\ell]{\lambda^{\epsilon}\mu} \notin L_{1}$. Therefore, $L_{1}\left(\sqrt[\ell]{\lambda^{\epsilon}\mu}\right)$ is a normal extension of $L_{1}$ of degree $\ell$ in $L_{N}$, since $L_{N} \supseteq \Fp^{\alg}$ contains a primitive $\ell$th root of unity. The normal extension corresponds to a normal subgroup of $r_{\cB_{N}}\left(L_{N}/L_{1}\right) = \SL{N}$ with index $\ell$. This contradicts Lemma~\ref{lem:SL_comm} since $\ell \nmid N$.

	To show \eqref{eqn:trinity}, we prove that 
	\begin{enumerate}[label=(\roman*)]
		\item the extensions $L_{N}\left(\sqrt[\ell]{\lambda}\right)$, $\Lint$, and $L_{N}\left(\sqrt[\ell]{\mu}\right)$ are normal extensions of $L_{N}$ of degree $\ell$,
		\item there exists a unique normal extension of $L_{N}$ of degree $\ell$ in $L_{N \ell}$, and
		\item there exists a unique normal extension of $L_{N}$ of degree $\ell$ in $L_{p^{e}N}$.
	\end{enumerate}
	We note that (i) and (ii) imply $L_{N}\left(\sqrt[\ell]{\lambda}\right) = \Lint$ since they are intermediate fields between $L_{N}$ and $L_{N \ell}$. Similarly, (i) and (iii) imply $L_{N}\left(\sqrt[\ell]{\mu}\right) = \Lint$.

	First, we show (i). By Lemma~\ref{lem:ind_hyp}, we conclude that $L_{\ell} \cap L_{N} = L_{1}$. Since $L_{N}$ is a Galois extension of~$L_{1}$ and $$L_{1} \subseteq L_{N}~\cap~L_{1}\left(\sqrt[\ell]{\lambda}\right) \subseteq L_{N}~\cap~L_{\ell} = L_{1},$$ we have that \begin{align*}
		\left[L_{N}\left(\sqrt[\ell]{\lambda}\right):L_{N}\right]
		 = \left[L_{1}\left(\sqrt[\ell]{\lambda}\right) : L_{1}\left(\sqrt[\ell]{\lambda}\right) \cap L_{N} \right] = \left[L_{1}\left(\sqrt[\ell]{\lambda}\right) : L_{1} \right]= \ell.
	\end{align*} 
	By the assumption that $\ell \nmid v\left(\mu\right)$, we have $\sqrt[\ell]{\mu} \notin L_{1}$. Similarly, we get that $\left[L_{N}\left(\sqrt[\ell]{\mu}\right):L_{N}\right] = \ell$. Since $L_{N} $ contains $ \Fp^{\alg}$, $L_{N}$ contains a primitive $\ell$th root of unity, so both $L_{N}\left(\sqrt[\ell]{\lambda}\right)$ and $L_{N}\left(\sqrt[\ell]{\mu}\right)$ are normal extensions over $L_{N}$ of degree $\ell$. Since $\cNintersec \ne \S{N \ell}{N}$, by Proposition~\ref{prop:tool}~\ref{cyc}, Lemma~\ref{lem:SLNN}~\ref{SL}, and Lemma~\ref{lem:SLl_comm}~\ref{SLl_comm_small_ell}, we see that $\cNintersec$ is a normal subgroup of $\S{N \ell}{N}$ with index~$\ell$. In other words, the field $\Lint$ is a normal extension of $L_{N}$ of degree~$\ell$.

	The part (ii) follows from Proposition~\ref{prop:tool}~\ref{cyc}, Lemma~\ref{lem:SLNN}~\ref{SL}, and Lemma~\ref{lem:SLl_comm}~\ref{SLl_comm_small_ell}.
	
	The part (iii) holds since $\fkA\left(L_{p^{e} N}/L_{N}\right)$ is cyclic and $\ell$ does not divide the inseparable degree of $L_{p^{e} N} $ over $L_{N}$.
\end{proof}

Our next objective is to establish the existence of $v$, $\lambda$, and $\mu$ as described in Proposition~\ref{prop:2or3}~\ref{goal_2or3}. To achieve this, we compute the valuations of certain elements in $L_{1}$, as detailed in the following sections.

\subsubsection{Basic properties of valuations}\label{subsubsec:val}

Throughout this section, let $F$ be a field with a valuation~$v$ which is not necessarily to be discrete. In other words, $v$ is a group homomorphism from $F^{\times}$ to a totally ordered additive group $\left(\Gamma,\le\right)$ such that $$
	v\left(a+b\right) \ge \min \left\{v\left(a\right), v\left(b\right)\right\} \text{ for all } a,b \in F^{\times} \text{ with  }a+b \ne 0.
$$ We extend $v$ from $F$ to $\Gamma \cup \{\infty\}$ for an element $\infty\not\in \Gamma$, with the relation $\infty > g$ for all $g \in \Gamma$, by defining $v\left(0\right) = \infty$. For all $a,b \in F$, one may verify the following useful equality: \begin{equation}\label{eqn:valuation}
	v\left(a+b\right) = \min \left\{v\left(a\right), v\left(b\right)\right\}, \text{ if } v\left(a\right) \ne v\left(b\right).
\end{equation}

\begin{lemma}\label{lem:tors-free}
	A given codomain $\Gamma$ of $v$ in the above is torsion-free.
\end{lemma}
\begin{proof}
	We show that if $g \in \Gamma$ has finite order $n$, then $g=0$. If $g>0$, we have a contradiction that $0 = ng>0$. Similarly, $g<0$ leads to a contradiction. Hence, $g=0$.
\end{proof}

Lemma~\ref{lem:tors-free} allows us to extend the codomain $\Gamma$ of $v$ given above to a $\bbq$-vector space $\bbq \otimes_{\bbz} \Gamma$. More precisely, the lemma asserts that the natural group homomorphism $\Gamma \to \bbq \otimes_{\bbz} \Gamma$ defined by $g \mapsto 1 \otimes_{\bbz} g$ is injective, so we can consider $\Gamma$ as a subgroup of $\bbq \otimes_{\bbz} \Gamma$. Moreover, the total order $\le$ on $\Gamma$ extends naturally to $\bbq \otimes_{\bbz} \Gamma$. In fact, if we denote each element $a \otimes_{\bbz} g \in \bbq \otimes_{\bbz} \Gamma$ simply by $ag$, then every element in $\bbq \otimes_{\bbz} \Gamma$ can be  written in the form $\frac{g}{n}$ for some $g \in \Gamma$ and some positive $n \in \bbz$. We define $\frac{g}{n} \le \frac{g'}{n'}$ if and only if $n' g \le n g'$ in $\Gamma$. 

For simplicity, we continue to denote the extended codomain $\bbq \otimes_{\bbz} \Gamma$ of $v$ by the same notation~$\Gamma$.

\

Throughout this paper, we treat an integer $k$ as an element of a field $F$ when necessary, defining $k := \begin{cases} \sum_{i=1}^{k} 1_{F}, &\text{ if } k> 0,\\
0_{F}, &\text{ if } k=0,\\
-\sum_{i=1}^{-k} 1_{F}, &\text{ if } k<0,
\end{cases}$ where $0_{F}$ and $1_{F}$ are the additive and multiplicative identities of $F$, respectively.
\begin{lemma}\label{lem:val}
	Let $q$ denote the characteristic of $F$ and $c\in F-\{0\}$. If $q=0$, we have $v\left(c\right) \ge 0$ if $c$ is integral over $\bbz$, and if $q>0$, we have $v\left(c\right) = 0$ if $c \in \bbf_{q}^{\alg}$.
\end{lemma}
\begin{proof}
	If $q=0$, then $v\left(a\right)\geq 0$ for any integer $a$. Indeed, if $a>0$, we have $$v\left(a\right) \ge v\left(\sum_{k=1}^{a}1\right) \ge v\left(1\right) = v\left(1/1\right) = 0.$$ If $a<0$, then $v\left(a\right)= v\left(-a\right) \ge 0$. Lastly, for $a=0$, we have $v\left(0\right)=\infty \ge 0$.
	
	If $c$ is integral over $\bbz$, then $c$ satisfies a monic polynomial $x^{d} + \sum_{i=0}^{d-1}a_{i}x^{d-i}$ with coefficients $a_{i} \in \bbz$. If $v\left(c\right) <0$, then \eqref{eqn:valuation} implies that $\infty = v\left(0\right) = v\left(c^{d} + \sum_{i=0}^{d-1}a_{i}c^{d-i}\right) = d v\left(c\right) < 0$, which is a contradiction. 
	
	Next, suppose $q>0$. If $c \in \bbf_{q}^{\alg}$, then  $c^{d} = 1$ for some positive integer $d$, so $dv\left(c\right) = v\left(1\right) = 0$. By Lemma~\ref{lem:tors-free}, it follows that $v\left(c\right) = 0$.
\end{proof}

\begin{lemma}\label{lem:disc}
	Let $f$ be a monic polynomial $f\left(x\right) = x^{n} + \sum_{i=1}^{n} c_{i} x^{n-i} \in F\left[x\right]$ of degree $n\ge2$. If $v\left(n\right) = 0$, and $v\left(c_{i}\right) > \frac{i}{n}v\left(c_{n}\right)$ for all $1 \le i\le n-1$, then $v\left(\Disc{f}\right) = \left(n-1\right) v\left(c_{n}\right)$, where $\Disc{f}$ denotes the discriminant of $f$.
\end{lemma}

\begin{proof}
	Let $T:=\{\left(d_{1},d_{2},\cdots,d_{n}\right) \in \bbz^{n} : d_{i}\geq 0, ~\sum_{1 \le i\le n} i d_{i} = n\left(n-1\right)\}$, which is a finite set. For indeterminates $a_{1}, a_{2}, \cdots, a_{n}$ and the monic polynomial $g\left(x\right) = x^{n} + \sum_{i=1}^{n} a_{i} x^{n-i} \in \left(\bbz\left[a_{1},a_{2},\cdots,a_{n}\right]\right)\left[x\right]$ in variable $x$, we can let that 
	\begin{align}\label{eq:disc}
		\Disc{g} = \sum\limits_{I \in T} r_{I} a^{I}, \text{ for some } r_{I} \in \bbz, \text{ where } a^{I}=\prod_{1 \le i \le n} a_{i}^{d_{i}}. 
	\end{align}
	In particular, the discriminant of the polynomial $x^{n}+1$ is $r_{\left(0,\cdots,0,n-1\right)}$. Since the zeros of $x^{n}+1$ are $\zeta_{2n}^{2i-1}$ for $i \in \left\{1,2,\cdots,n\right\}$, where $\zeta_{2n}$ is a $2n$-th primitive root of unity, we derive: \begin{align*}
		r_{\left(0,\cdots,0,n-1\right)}
		& = \left(\prod_{1\leq i<j\leq n}\left(\zeta_{2n}^{2i-1} - \zeta_{2n}^{2j-1}\right)\right)^{2} \\
		\notag & = \left(-1\right)^{\frac{n\left(n-1\right)}{2}}\prod_{i=1}^{n}\left(\prod_{j \ne i}\left(\zeta_{2n}^{2i-1} - \zeta_{2n}^{2j-1}\right)\right) \\
		\notag & = \left(-1\right)^{\frac{n\left(n-1\right)}{2}} \prod_{i=1}^{n}\left(\left.\frac{d}{dx}\left(x^{n}+1\right)\right|_{x=\zeta_{2n}^{2i-1}}\right) \\
		\notag & = \left(-1\right)^{\frac{n\left(n-1\right)}{2}} \prod_{i=1}^{n}\left(n\zeta_{2n}^{\left(n-1\right)\left(2i-1\right)}\right) \\
		\notag & = \left(-1\right)^{\frac{n\left(n-1\right)}{2}} n^{n}\left(\zeta_{2n}^{\left(n-1\right)n^{2}}\right) \\
		\notag & = \left(-1\right)^{\frac{n\left(n-1\right)}{2}} n^{n},
	\end{align*} and so $v\left(r_{\left(0,\cdots,0,n-1\right)}\right) = 0$ since $v\left(n\right)=0$.
	
	Now, for given $f$ and $r_{I}$ described in \eqref{eq:disc}, let $\Disc{f} = \sum\limits_{I \in T} r_{I} c^{I}$ where $c^{I}=\prod_{1 \le i \le n} c_{i}^{d_{i}}$. Then, for each $I = \left(d_{1},d_{2},\cdots,d_{n}\right) \ne \left(0,\cdots,0,n-1\right)$ in $T$, there exists $1\le k_{I}\le n-1$ with $d_{k_{I}} > 0$. By Lemma~\ref{lem:val} recalling $v\left(0\right) = \infty \ge 0$, we have $v\left(r_{I}\right) \ge 0$ for all $I$. By the assumption $v\left(c_{i}\right) > \frac{i}{n} v\left(c_{n}\right)$ for all $i<n$, we have: \begin{align*}
		v\left(r_{I}c^{I}\right) \ge v\left(c^{I}\right) = d_{k_{I}}v\left(c_{k_{I}}\right) + \left(\sum_{i \ne k_{I}} d_{i}v\left(c_{i}\right)\right)
		& > \frac{k_{I}d_{k_{I}}}{n}v\left(c_{n}\right) + \left(\sum_{i \ne k_{I}} d_{i} v\left(c_{i}\right)\right) \\
		\notag & \ge \frac{k_{I}d_{k_{I}}}{n}v\left(c_{n}\right) + \left(\sum_{i \ne k_{I}} \frac{id_{i}}{n}v\left(c_{n}\right)\right) \\
		\notag & = \left(n-1\right)v\left( c_{n}\right).
	\end{align*} Moreover, $v\left(r_{\left(0,\cdots,0,n-1\right)}c^{\left(0,\cdots,0,n-1\right)}\right) = v\left(c^{\left(0,\cdots,0,n-1\right)}\right) = \left(n-1\right) v\left(c_{n}\right)$. So, \eqref{eqn:valuation} implies that $v\left(\Disc{f}\right) = v\left(\sum_{I \in T} r_{I} c^{I}\right) = \left(n-1\right)v\left(c_{n}\right)$.
\end{proof}

\begin{lemma}\label{lem:cubic}
	Suppose that the characteristic of $F$ is $\geq 5$ and $F$ contains $\sqrt{3}$ and a primitive cube root of unity $\omega$. If the Galois group of an irreducible cubic polynomial $f\left(x\right) = x^{3} + c_{1}x^{2} + c_{2}x + c_{3} \in F\left[x\right]$ is cyclic and $v\left(c_{1}\right) > \frac{v\left(c_{3}\right)}{3}$ and $v\left(c_{2}\right) > \frac{2v\left(c_{3}\right)}{3}$, then the splitting field of $f$ over $F$ is generated by a cube root of $\beta \in F$ such that $v\left(\beta\right) = v\left(c_{3}\right)$.
\end{lemma}
\begin{proof}
	We let $A := \frac{3c_{2}-c_{1}^{2}}{3}$ and $B := \frac{2c_{1}^{3}-9c_{1}c_{2}+27c_{3}}{27}$, so that $f\left(x-c_{1}/3\right)=x^{3} + Ax + B$. Since the Galois group of $f$ over $F$ is a cyclic transitive subgroup of the symmetric group of degree $3$, it is the alternating group of degree $3$. Hence, the discriminant $4A^{3}+27B^{2}$ of $f$ is a square in $F$. Let $R\in F$ be a square root of $4A^{3}+27B^{2}$. By Cardano's formula, the zeros of $f$ are $$
		\frac{c_{1}}{3} +\omega^{i} \sqrt[3]{\beta} - \frac{A/3}{\omega^{i}\sqrt[3]{\beta}}, \text{ for } i \in \left\{0,1,2\right\},
	$$ where $\beta$ is either $\beta_{+} := - \frac{B}{2} + \frac{R}{6\sqrt{3}}$ or $\beta_{-} : = - \frac{B}{2} - \frac{R}{6\sqrt{3}}$. It is easy to show that the splitting field of~$f$ is generated by $\sqrt[3]{\beta}$. 
	
	We show that the $v$-valuation of either $\beta_{+}$ or $\beta_{-}$ is $v\left(c_{3}\right)$. From the assumption that $v\left(c_{1}\right) > \frac{v\left(c_{3}\right)}{3}$ and $v\left(c_{2}\right) > \frac{2v\left(c_{3}\right)}{3}$, it follows that $2v\left(R\right) = v\left(\operatorname{Disc}\left(f\right)\right) = 2v \left(c_{3}\right)$ by Lemma~\ref{lem:disc}, and $v \left(B\right) = v\left(c_{3}\right)$ by~\eqref{eqn:valuation} and Lemma~\ref{lem:val}. By Lemma~\ref{lem:tors-free}, $ v\left(R\right) = v \left(c_{3}\right) = v\left(B\right)$. We also have $v\left(2\right)=v\left(6\sqrt{3}\right)=0$, by Lemma~\ref{lem:val}. Therefore, both $v\left( \beta_{+} \right)$ and $ v\left( \beta_{-} \right)$ satisfy $v\left( \beta_{+} \right), v\left( \beta_{-} \right) \ge v\left(c_{3}\right)$. If the $v$-valuations of both $\beta_{+}$ and $\beta_{-}$ are strictly greater than $v\left(c_{3}\right)$, then $v\left(c_{3}\right) = v\left(B\right) \ge v\left(\beta_{+} + \beta_{-}\right) > v\left(c_{3}\right)$, which is a contradiction. Thus, at least one of $v\left(\beta_{+}\right)$ or $v\left(\beta_{-}\right)$ must be equal to $v\left(c_{3}\right)$.
\end{proof}

The existence of such a  valuation $v$, as described in Proposition~\ref{prop:2or3}~\ref{goal_2or3} is established in \cite[the proof of Proposition~3.8]{IK24}. Once such a valuation $v$ is fixed, several important properties of $v$ are related to the polynomial coefficients $a_{k} \in \bbf_{p}\left[s,t\right] \subseteq L_{1}$ of the $p$th division polynomial (\cite[Lemma~2.5]{IK24}), \begin{equation}\label{theta}
	\theta_{s,t}\left(X\right) = a_{\frac{p-1}{2}} X^{\frac{p-1}{2}} + \left(\sum_{i=1}^{\frac{p-3}{2}} a_{\frac{p-1}{2}+ pi} X^{\frac{p-1}{2}-i}\right) + a_{\frac{p^{2}-1}{2}} \in \left(\bbf_{p}\left[s,t\right]\right)\left[X\right]
\end{equation} of the elliptic curve $y^{2} = x^{3} + sx +t$ over the function field $\Fp\left(s,t\right)$. We now list the relevant properties.

\begin{lemma}\label{lem:IK24} There exist discrete valuations $v$ on $L_{1}$ and $w$ on $L_{1}\left(x\left(E_{s,t}\left[p\right]\right)\right)$ such that 
	\begin{enumerate}[\normalfont (a)]
		\item\label{Kst} $v\left(c\right) \ge 0$ for all $c \in \bbk\left[s,t\right]$,
		\item\label{coeff} $v\left(a_{\frac{p-1}{2}+ pi}/a_{\frac{p-1}{2}}\right) \ge 0$ for all $0 \le i \le \frac{p-3}{2}$ and $v\left(a_{\frac{p^{2}-1}{2}}/a_{\frac{p-1}{2}}\right) = -1$,
		\item\label{zero} $w\left(u\right) = - 1$ for the $x$-coordinate $u$ of a point in $E_{s,t}\left[p\right]$ of order $p$,
		\item\label{ram_idx} $\left.w\right|_{L_{1}} = p\frac{p-1}{2}v$, and
		\item\label{disc} $v\left(4s^{3}+27t^{2}\right) = 0$.
	\end{enumerate}
\end{lemma}

\begin{proof}
	The fields $L_{1}$ and $L_{1}\left(x\left(E_{s,t}\left[p\right]\right)\right)$ correspond to $F_{0}$ and $F_{1}$ in \cite[Proposition~3.8]{IK24}, respectively. We define our valuations $v$ and $w$ in the proof of \cite[Proposition~3.8]{IK24} for the case $k=1$. The Dedekind domain $R_{0}$ given in \cite[(11)]{IK24} contains $\bbk\left[s,t\right]$, so the part \ref{Kst} follows immediately.
	
	The part \ref{coeff} follows from \cite[Corollary~3.4]{IK24}, and the part~\ref{zero} follows from \cite[Proposition~3.8]{IK24} for $k=1$. The proof of \ref{ram_idx} is also given in the proof of \cite[Proposition~3.8]{IK24} for $k=1$.
	
	To prove \ref{disc}, we refer to \cite[(9), Proposition~3.3, and (3)]{IK24}. By the definition of $v$ in the proof of \cite[Proposition~3.8]{IK24} for $k=1$, the $v$-valuation of at least one among $s$, $t$, or $s^{3} - j_{\operatorname{ss}} \frac{s^{3} + 27t^{2}/4}{1728}$ is positive, where $j_{\ss} \ne 0,1728$ is the $j$-invariant of a supersingular elliptic curve over $\bbf_{p}^{\operatorname{alg}}$, noting that $\bbk \supseteq \bbf_{p}^{\operatorname{alg}}$. By \ref{Kst}, $v\left(4s^{3}+27t^{2}\right) \ge 0$. If $v\left(4s^{3}+27t^{2}\right) > 0$, then both $s$ and $t$ must have positive $v$-valuation, contradicting the fact that either $s$ or $t$ is a unit in~$R_{0}$, which completes the proof of \ref{disc}.
\end{proof}

\subsubsection{Proof of $\cNintersec = \S{N\ell}{N}$ for the cases when $\ell \in \left\{2,3\right\}$ does not divide $N$}\label{subsubsec:2or3}
In this section, we show the existence of $\lambda$ and $\mu$ in Proposition~\ref{prop:2or3}~\ref{goal_2or3} with respect to the valuation~$v$ in~Lemma~\ref{lem:IK24}, and complete the proof of $\cN_{p^{e}, \cB_{\ell N}, \ell} = \S{\ell N}{N}$ in the cases when $\ell \in \left\{2,3\right\}$ does not divide $N$.

\begin{lemma}\label{lem:gen}
	For $\ell \in \left\{2,3\right\}$, we have $\sqrt[\ell]{4s^{3}+27t^{2}} \in L_{\ell}$.
\end{lemma}
\begin{proof}
	For $\ell=2$, $L_{2}$ is the splitting field of the cubic polynomial $x^{3} + sx + t$ over $L_{1}$, so it contains $\sqrt{4s^{3} + 27t^{2}}$, which is a square root of the discriminant of $x^{3} + sx + t$.
	
	We show that $\sqrt[3]{4s^{3}+27t^{2}} \in L_{3}$. For a basis $\left\{P_{3},Q_{3}\right\}$ of $E_{s,t}\left[3\right]$, we consider the cubic polynomial, \begin{align*}
		f\left(T\right)
		:= & \left(T + \left(x\left(P_{3}\right)x\left(Q_{3}\right) + x\left(P_{3} - Q_{3}\right)x\left(P_{3} + Q_{3}\right)\right)\right) \\
		\notag & \phantom{=} \times \left(T + \left(x\left(P_{3}\right)x\left(P_{3}-Q_{3}\right) + x\left(Q_{3}\right)x\left(P_{3}+Q_{3}\right)\right)\right) \\
		\notag & \phantom{=} \times \left(T + \left(x\left(P_{3}\right)x\left(P_{3}+Q_{3}\right) + x\left(Q_{3}\right)x\left(P_{3}-Q_{3}\right)\right)\right) \\
	\end{align*} defined over $L_{3}$. It is easy to show that $f$ is invariant under $\Gal\left(L_{3}/L_{1}\right)$, so $f\left(T\right)\in L_{1}\left[T\right]$. To compute the coefficients of $f$ explicitly, we use Vieta's formulas for the 3rd division polynomial, $3x^{4} + 6s x^{2} + 12t x - s^{2}$ of $E_{s,t}$. The zeros of the 3rd division polynomial are the $x$-coordinates of points in $E_{s,t}$ of order $3$ (\cite[Ch.III, Exercise 3.7]{Silverman}). For $1\le k \le 4$, if $\mathcal{S}_{k}$ denotes the elementary symmetric polynomials of degree $k$ in $\alpha_{1} := x\left(P_{3}\right)$, $\alpha_{2} := x\left(Q_{3}\right)$, $\alpha_{3} := x\left(P_{3} - Q_{3}\right)$, and $\alpha_{4} := x\left(P_{3} + Q_{3}\right)$, then Vieta's formulas for the 3rd division polynomial lead: $$
		\mathcal{S}_{1} = 0, \mathcal{S}_{2} = 2s, \mathcal{S}_{3} = -4t \text{, and } \mathcal{S}_{4} = -s^{2}/3.
	$$

	The $T^{2}$-coefficient of $f$ is $$
		\sum_{i<j} \alpha_{i}\alpha_{j} = \mathcal{S}_{2} = 2s,
	$$ 
	its $T$-coefficient is \begin{align*}
		\prod_{\substack{i<j \\ k\ne i,j}}\alpha_{i}\alpha_{j}\alpha_{k}^{2}
		= \left(\sum_{t}\alpha_{t}\right) \left(\sum_{i<j<k}\alpha_{i}\alpha_{j}\alpha_{k}\right) - 4\alpha_{1}\alpha_{2}\alpha_{3}\alpha_{4} = \mathcal{S}_{1}\mathcal{S}_{3} - 4\mathcal{S}_{4}= \frac{4}{3}s^{2},\\
	\end{align*}
	and its constant term is \begin{align*}
		\left(\alpha_{1}\alpha_{2}+\alpha_{3}\alpha_{4}\right) \left(\alpha_{1}\alpha_{3}+\alpha_{2}\alpha_{4}\right) \left(\alpha_{1}\alpha_{4}+\alpha_{2}\alpha_{3}\right)
		& = \left(\alpha_{1}\alpha_{2}\alpha_{3}\alpha_{4}\right) \left(\sum_{i} \alpha_{i}^{2} \right) + \left( \sum_{i<j<k}\alpha_{i}^{2}\alpha_{j}^{2}\alpha_{k}^{2} \right)\\
		\notag & = \mathcal{S}_{4} \left( \mathcal{S}_{1}^{2} - 2 \mathcal{S}_{2} \right) + \mathcal{S}_{3}^{2} - 2\left(\alpha_{1}\alpha_{2}\alpha_{3}\alpha_{4}\right) \left( \sum_{i<j}\alpha_{i}\alpha_{j} \right)\\
		\notag & = \mathcal{S}_{4} \left( \mathcal{S}_{1}^{2} - 2 \mathcal{S}_{2} \right) + \mathcal{S}_{3}^{2} - 2 \mathcal{S}_{4}\mathcal{S}_{2}\\
		\notag & = \frac{8}{3}s^{3}+16t^{2}.\\
	\end{align*}
	A direct computation shows that $f\left(T-2s/3\right) = T^{3} + \frac{16}{27} \left(4s^{3}+27t^{2}\right)$. Since $\sqrt[3]{16/27} \in \Fp^{\alg} \subseteq L_{1}$, the splitting field of $f$ over $L_{1}$ contains $\sqrt[3]{4s^{3}+27t^{2}}$, and the splitting field of $f$ is contained in $L_{3}$ by the definition of $f$.
\end{proof}

For the discrete valuation $v$ given in Lemma~\ref{lem:IK24}, $\lambda := 4s^{3}+27t^{2} \in L_{1}$ satisfy the conditions $\sqrt[\ell]{\lambda} \notin L_{1}$, $\sqrt[\ell]{\lambda} \in L_{\ell}$, and $\ell \mid v\left(\lambda\right)$ in Proposition~\ref{prop:2or3}~\ref{goal_2or3}. The last step to prove $\cNintersec = \S{N \ell}{N}$ in Figure~\ref{fig:fld_extn} for the cases when $\ell \in \left\{2,3\right\}$ does not divide $N$ is finding $\mu \in L_{1}$ such that $\ell \nmid v\left(\mu\right)$ and $\sqrt[\ell]{\mu} \in L_{\ell}$. Such $\mu$ is given in the poof of Proposition~\ref{prop:2} for $\ell = 2$, or Proposition~\ref{prop:3} for $\ell=3$.

\begin{prop}\label{prop:2}
	Let $N$ be a positive integer such that $2 \nmid N$. If $G_{e, \cB_{N}} = \peu \times \SL{N}$ for a basis $\cB_{N}$ of $E_{s,t}\left[N\right]$, then $\cN_{p^{e}, \cB_{2 N}, 2} = \S{2 N}{N}$ for any basis $\cB_{2N}$ of $E_{s,t}\left[2N\right]$ referring to Figure~\ref{fig:fld_extn}.
\end{prop}
\begin{proof}
	Let $\lambda := 4s^{3} + 27t^{2}$. For the discrete valuation $v$ given in Lemma~\ref{lem:IK24}, we have that $\sqrt{\lambda} \in L_{2}$ and $2 \mid v\left(\lambda\right)$, by Lemma~\ref{lem:gen} and Lemma~\ref{lem:IK24}~\ref{disc}. Obviously, $\sqrt{\lambda} \notin\Kst=L_{1}$. If we find a $\mu \in L_{1}$ such that $\sqrt{\mu} \in L_{p}$ and $v\left(\mu\right)$ is odd, then we complete the proof by Proposition~\ref{prop:2or3}~\ref{goal_2or3}.
	
	For $p \equiv 1 \pmod{4}$, we let $\mu$ be the discriminant of the monic polynomial $\frac{1}{a_{\frac{p-1}{2}}}\theta_{s,t}\left(X\right) \in L_{1}\left[X\right]$ given by \eqref{theta}. Since the zeros of $\frac{1}{a_{\frac{p-1}{2}}}\theta_{s,t}\left(X\right)$ are the $p$th powers of the $x$-coordinates of points in $E_{s,t}$ of order $p$, we have that $\sqrt{\mu} \in L_{1}\left(x\left(E_{s,t}\left[p\right]\right)\right) \subseteq L_{p}$. By Lemma~\ref{lem:IK24}~\ref{coeff} and Lemma~\ref{lem:disc}, $v\left(\mu\right)=- \frac{p-3}{2}$, which is odd.

	If $p \equiv 3 \pmod{4}$, then by Theorem~\ref{thm:fact:pe} and \cite[Proposition~3.10]{IK24}, the extension $L_{p} / L_{1}$ is normal with the cyclic automorphism group $\pu$ of order $p-1$ and inseparable degree~$p$. Hence, there is an element $\mu \in L_{1}$ such that $\sqrt{\mu} \in L_{p}$. By \cite[Corollary~3.7]{IK24}, it is known that the degree of the extension $L_{1}\left(x\left(E_{s,t}\left[p\right]\right)\right) $ over $ L_{1}$ is $p\frac{p-1}{2}$, which is odd. Therefore, we have that $\sqrt{\mu} \notin L_{1}\left(x\left(E_{s,t}\left[p\right]\right)\right)$ and conclude that $L_{p} = L_{1}\left(x\left(E_{s,t}\left[p\right]\right),\sqrt{\mu}\right)$.

	Therefore, for the $x$-coordinate $u$ of a point in $E_{s,t}\left[p\right]$ of order $p$, there exist $a,b \in L_{1}\left(x\left(E_{s,t}\left[p\right]\right)\right)$ such that $\left(u,a+b\sqrt{\mu}\right) \in E_{s,t}\left[p\right]$. We note that $b \ne 0$, since $L_{1}\left(x\left(E_{s,t}\left[p\right]\right)\right) \ne L_{p}$. Thus, the equality $\left(a+b\sqrt{\mu}\right)^{2} = u^{3} + su + t$ implies $a=0$. For the discrete valuation~$w$ given in Lemma~\ref{lem:IK24}, from $b^{2} \mu = u^{3} + su + t$, we deduce $w\left(b^{2} \mu\right) = w\left( u^{3} + su + t \right) = -3$ by Lemma~\ref{lem:IK24}~\ref{zero}, Lemma~\ref{lem:IK24}~\ref{Kst}, and the equality~\eqref{eqn:valuation}, so $w\left(\mu\right) \equiv w\left(b^{2} \mu\right) \equiv 1 \pmod{2}$. Finally, Lemma~\ref{lem:IK24}~\ref{ram_idx} shows that $v\left(\mu\right)$ is odd.
\end{proof}

Now, we consider when $\ell=3$.
\begin{prop}\label{prop:3}
	Let $N$ be a positive integer such that $3 \nmid N$. If $G_{e, \cB_{N}} = \peu \times \SL{N}$ for any $\cB_{N}$ of $E_{s,t}\left[N\right]$, then $\cN_{p^{e}, \cB_{3 N}, 3} = \S{3 N}{N}$ for any basis $\cB_{3N}$ of $E_{s,t}\left[3N\right]$ referring to Figure~\ref{fig:fld_extn}.
\end{prop}
\begin{proof}
	If $p \equiv 2\pmod{3}$, the proof follows from Proposition~\ref{prop:2or3}~\ref{easy}. 
	
	Suppose that $p \equiv 1\pmod{3}$. Let $\lambda := 4s^{3} + 27t^{2}$. For the discrete valuation $v$ given in Lemma~\ref{lem:IK24}, we have that $\sqrt[3]{\lambda} \in L_{3}$ and $3 \mid v\left(\lambda\right)$, by Lemma~\ref{lem:gen} and Lemma~\ref{lem:IK24}~\ref{disc}. Obviously, $\sqrt[3]{\lambda} \notin \Kst = L_{1}$. Then, by Proposition~\ref{prop:2or3}~\ref{goal_2or3}, it is enough to find $\mu \in L_{1}$ such that $\sqrt[3]{\mu} \in L_{p}$ and $v\left(\mu\right)=-1$, to complete the proof.

	For a generator $R$ of $E_{s,t}\left[p\right]$ and a generator $\alpha$ of $\pu \cong \fkA\left(L_{p}/L_{1}\right)$, a direct computation shows that the cubic polynomial $$
		f\left(T\right)
		:= T^{3} + c_{1} T^{2} + c_{2} T + c_{3}
		:= \prod_{r=0}^{2} \left(T + \left(\prod_{k=0}^{\frac{p-1}{6}-1}x\left(\alpha^{r+3k}R\right)\right)^{p} \right)
	$$ defined over $L_{p}$ is invariant under $\fkA\left(L_{p}/L_{1}\right)$. Moreover, the inseparable degree of $L_{p} / L_{1}$ is $p$ by \cite[Proposition~3.10]{IK24}. Therefore, the coefficients $c_{1}, c_{2}, c_{3} \in L_{p}$ are separable over $L_{1}$ and are invariant under $\fkA\left(L_{p}/L_{1}\right)$. This implies $c_{1}, c_{2}, c_{3} \in L_{1}$. For the discrete valuation $w$ given in Lemma~\ref{lem:IK24}, we have $w\left(x\left(\alpha^{i}R\right)\right) = -1$ for all $0 \le i \le \frac{p-1}{2}-1$, by Lemma~\ref{lem:IK24}~\ref{zero}. Thus, $w\left(c_{1}\right) \ge -p\frac{p-1}{6}$, $w\left(c_{2}\right) \ge -p\frac{p-1}{3}$, and $w\left(c_{3}\right) = -p\frac{p-1}{2}$. Since $c_{1}, c_{2}, c_{3} \in L_{1}$, we conclude that $v\left(c_{1}\right) \ge -\frac{1}{3}$, $v\left(c_{2}\right) \ge -\frac{2}{3}$, and $v\left(c_{3}\right) = -1$ by Lemma~\ref{lem:IK24}~\ref{ram_idx}. Since the $v$-valuations are integers, we deduce:\begin{equation}\label{eqn:coef_ineq}
		v\left(c_{1}\right)\ge 0 > \frac{1}{3}v\left(c_{3}\right) \text{ and } v\left(c_{2}\right)\ge 0 > \frac{2}{3}v\left(c_{3}\right).
	\end{equation}
	Therefore, by Eisenstein's criterion, the polynomial $U^{3} + \frac{c_{2}}{c_{3}} U^{2} + \frac{c_{1}}{c_{3}} U + \frac{1}{c_{3}}$ of variable $U$ is irreducible over $\left\{a\in L_{1}: v\left(a\right) \ge 0\right\}$ which is the DVR (discrete valuation ring) with respect to the discrete valuation $v$. By Gauss's lemma (\cite[§~9.~Proposition~5]{Dummit-Foote}), $f$ is irreducible over~$L_{1}$ which is the field of fraction of this DVR. Since each zero of $f$ is fixed by $\alpha^{3}$, the Galois group of $f$ over $L_{1}$ is a quotient group of the cyclic group $\left\langle \alpha \right\rangle / \left\langle \alpha^{3} \right\rangle$. By Lemma~\ref{lem:cubic} and \eqref{eqn:coef_ineq}, the splitting field of $f$ over $L_{1}$ is generated by a cube root of $\mu \in L_{1}$ such that $v\left(\mu\right) = v\left(c_{3}\right) = -1$. Since the splitting field of $f$ over $L_{1}$ is contained in $L_{p}$, it follows that $\sqrt[3]{\mu} \in L_{p}$.
\end{proof}

\section{Proofs of main results}\label{sec:thm_and_cor}

We are ready to prove Proposition~\ref{prop:induction} and Theorem~\ref{thm:main} by synthesizing the results from the previous sections.

\begin{proof}[Proof of Proposition~\ref{prop:induction}]
	By Proposition~\ref{prop:general}, Proposition~\ref{prop:2}, and Proposition~\ref{prop:3}, we have $\cNintersec = \S{N \ell}{N}$ referring to Figure~\ref{fig:fld_extn} for any prime $\ell$ and any basis $\cB_{N}$ of $E_{s,t}\left[N\right]$. Proposition~\ref{prop:tool}~\ref{goal} completes the proof.
    
\end{proof}

\begin{proof}[Proof of Theorem~\ref{thm:main}]
	By Proposition~\ref{prop:induction} and mathematical induction, it suffices to prove the statement only for $N=1$. This base case when $N=1$ is exactly Theorem~\ref{thm:fact:pe}.
    
\end{proof}

Now, we prove Corollary~\ref{cor:gen}, which determines the automorphism group of $p^{e} N$-torsion subgroup of $E_{s,t}$, where $E_{s,t}$ is defined over a field not necessarily containing $\Fp^{\alg}$:

\begin{proof}[Proof of Corollary~\ref{cor:gen}]
	Via the Weil pairing proves, we have that $$
		G_{e, \cB_{N}} \subseteq \left\{\left(u,A\right) \in \peu \times \GL{N}: \det A \in \Im\left(\cyc\right)\right\}.
	$$ 
	Next, we show the reverse inclusion `$\supseteq$'.
	For each $\left(u,A\right) \in \peu \times \GL{N}$ such that $\det A \in \Im\left(\cyc\right)$, there exists $\sigma \in \Gal\left(K\left(\zeta_{N}\right)/K\right)$ such that $\zeta_{N}^{\sigma} = \zeta_{N}^{\det A}$, where $\zeta_{N}\in \Fp^{\alg}$ is a primitive $N$th root of unity. Via the Weil pairing, we see that $\zeta_{N} \in K\left(E_{s,t}\left[p^{e} N\right]\right)$. Hence, there exists $\widetilde{\sigma} \in \fkA\left(K\left(E_{s,t}\left[p^{e} N\right]\right)/K\right)$ such that $\left.\widetilde{\sigma}\right|_{K\left(\zeta_{N}\right)} = \sigma$. Let $\left(v,B\right) := \rho_{e, \cB_{N}}\left(\widetilde{\sigma}\right)$. We let~$\bbk$ be the composite field of $K$ and $\Fp^{\alg}$. By Theorem~\ref{thm:main}, we have \begin{align*}
		\left(uv^{-1}, AB^{-1}\right)
		& \in \peu \times \SL{N} \\
		\notag & = \rho_{e, \cB_{N}}\left( \fkA\left(\Kst\left(E_{s,t}\left[p^{e} N\right]/\Kst\right)\right)\right) \\
		\notag & \subseteq \rho_{e, \cB_{N}}\left( \fkA\left(K\left(s,t\right)\left(E_{s,t}\left[p^{e} N\right]/K\left(s,t\right)\right)\right)\right).
	\end{align*} Hence, there exists $\tau \in \fkA\left(K\left(s,t\right)\left(E_{s,t}\left[p^{e} N\right]/K\left(s,t\right)\right)\right)$ such that $\rho_{e, \cB_{N}}\left(\tau\right) = \left(uv^{-1}, AB^{-1}\right)$ and $\left(u,A\right) = \rho_{e, \cB_{N}} \left( \tau\widetilde{\sigma} \right) \in G_{e, \cB_{N}}$.
    
\end{proof}

\end{document}